\documentclass{amsart}
\usepackage[all]{xy}
\usepackage{amscd}
\usepackage[T1]{fontenc}
\usepackage{amsmath}
\usepackage{amssymb}
\usepackage{latexsym}
\usepackage{enumerate}

\begin{document}
 
%%%%%%%%%%% Numbered objects of "theorem" style (text italicized)
\newtheorem{theorem}{Theorem}[section]
\newtheorem{proposition}[theorem]{Proposition}
\newtheorem{lemma}[theorem]{Lemma}
\newtheorem{corollary}[theorem]{Corollary}
\newtheorem{conjecture}[theorem]{Conjecture}
\newtheorem{question}[theorem]{Question}

%%%%%%%%%%% Numbered objects of "non-theorem" style (text roman)
\theoremstyle{remark}
\newtheorem{remark}[theorem]{Remark}
\newtheorem{remarks}[theorem]{Remarks}
\newtheorem{example}[theorem]{Example}
\newtheorem{examples}[theorem]{Examples}
\newtheorem{ex}[theorem]{Example}

%%%%%%%%%%% Typography
\newcommand{\Aff}{\mathbb{A}}
\newcommand{\FF}{\mathbb{F}}
\newcommand{\GG}{\mathbb{G}}
\newcommand{\NN}{\mathbb{N}}
\newcommand{\PP}{\mathbb{P}}
\newcommand{\mfa}{\mathfrak{a}}
\newcommand{\mfb}{\mathfrak{b}}
\newcommand{\mfj}{\mathfrak{j}}
\newcommand{\mfp}{\mathfrak{p}}
\newcommand{\mfr}{\mathfrak{r}}
\newcommand{\Fq}{\FF_{q}}
\newcommand{\Fqb}{\overline{\FF}_{q}}

%%%%%%%%%%% Specific macros
\newcommand{\card}[1]{\lvert#1\rvert}
\newcommand{\idest}{\textit{i. e.} }
\newcommand{\set}[2]{\left\{#1\,\mid\,#2\right\}}

\newcommand{\codim}{\mathop{\mathrm{codim}}\nolimits}
\newcommand{\cumdeg}{\mathop{\hbox{\rm{c-deg}}}\nolimits}
\newcommand{\Gal}{\mathop{\mathrm{Gal}}\nolimits}
\newcommand{\lcm}{\mathop{\mathrm{lcm}}\nolimits}
\newcommand{\Sing}{\mathop{\mathrm{Sing}}\nolimits}
\newcommand{\Tr}{\mathop{\mathrm{Tr}}\nolimits}

\renewcommand{\mod}{\mathop{\mathrm{mod}}\nolimits}

%%%%%%%%%%%%%%%%%%%%%%%%%%%%%%%%%%%
%%%% Front matter

\title{On the Number of Points 
of Algebraic Sets over Finite Fields}

\author{Gilles Lachaud}
\address{Gilles Lachaud, Universit\'e d'Aix-Marseille \\
Institut de Math\'{e}matiques
de Marseille\\ Luminy case 907, F13288 Marseille cedex 9\\ France}
\email{gilles.lachaud@univ-amu.fr}

\author{Robert Rolland}
\address{Robert Rolland, Universit\'e d'Aix-Marseille \\
Institut de Math\'{e}matiques
de Marseille\\ Luminy case 907, F13288 Marseille cedex 9\\ France}
\email{robert.rolland@acrypta.fr}

\date{July 24, 2014}

\begin{abstract}
We determine upper bounds on the number of rational points of an affine or
projective algebraic set defined over an extension of a finite field by a system of polynomial equations,
including the case where the algebraic set is not defined over the finite field by itself.
A special attention is given to irreducible but not absolutely irreducible 
algebraic sets, which satisfy better bounds. We study the case of complete intersections,
for which we give a decomposition, coarser than the decomposition in irreducible components, but more directly 
related to the polynomials defining the algebraic set.
We describe families of algebraic sets having the maximum number of rational points in the affine case, 
and a large number of points in the projective case.

\medskip

\noindent Nous d\'eterminons des majorations du nombre de points d'un ensemble alg\'ebrique affine ou projectif, d\'efini
sur une extension d'un corps fini par un syst\`eme d'\'equations polynomiales, y compris
dans le cas o\`u l'ensemble alg\'ebrique n'est pas d\'efini sur le corps fini lui-m\^eme.
Une attention particuli\`ere est port\'ee aux ensemble alg\'ebriques irr\'eductibles mais non absolument
irr\'eductibles, pour lesquels nous obtenons de meilleures bornes. Nous \'etudions le cas des intersections compl\`etes,
pour lesquelles nous construisons une d\'ecomposition moins fine que la d\'ecomposition en composantes
irr\'eductibles, mais plus directement li\'ee aux polyn\^omes qui d\'efinissent l'ensemble alg\'ebrique.
Enfin, nous construisons des familles d'ensembles alg\'ebriques atteignant le nombre maximum
de points rationnels dans le cas affine, et comportant de nombreux points dans le cas projectifs.
\end{abstract}


\maketitle

%%%%%%%%%%%%%%%%%%%%%%%%%%%%%%%%%%%
%%%% Main matter

\section*{Introduction}

Let $X$ be an algebraic subset of the affine or projective space, defined over an extension of a given 
finite field ${\mathbb F}_q$. Our purpose is to give several bounds on the maximum number of points 
of $X$, with coordinates in ${\mathbb F}_q$ (unless explicitly stated, we do not assume that $X$ is 
defined over ${\mathbb F}_q$). These bounds are expressed in terms of the degree of $X$, and they are 
obtained by applying various versions of B\'ezout's Theorem. Hence, the notions of degree and cumulative 
degree are essential tools in our computations, and they are introduced in Section \ref{sec_cd}.  
We establish a general upper bound in Section \ref{sec_b} (Theorem \ref{genbound}). We improve 
this general bound if $X$ is irrational, that is, not defined over ${\mathbb F}_q$, 
by introducing the \emph{greatest $k$-closed subset} in $X$. Surprisingly, we obtain in 
this case a bound of order $q^{\dim X - 1}$ (Corollary \ref{shortbound}). 
In Section \ref{sec_rci} we assume that $X$ is relatively irreducible, and study 
the decomposition of $X$ in absolutely irreducible components of $X$. 
This decomposition leads to a better upper bound (Corollary \ref{relirrbound})  
than the general upper bound given in Section \ref{sec_b}, and also to a bound of 
order $q^{\dim X - 1}$. We also show that the set of rational points of $X$ is 
contained in the singular locus of $X$, and moreover $X(k) = \emptyset$ if $X$ is normal (Proposition \ref{singbound}).
We assume in section \ref{sec_ci} that $X$ is an (ideal-theoretic) complete intersection, 
for which an exact formula for the degree is known.
We describe a decomposition of $X$ directly related to a system $(f_{1}, \dots, f_{r})$ of 
polynomials defining $X$, namely the \emph{coarse decomposition} (Proposition \ref{GenBasic}). 
The decomposition in irreducible components is finer
than the coarse decomposition, but the later can be explicitely constructed
from $(f_{1}, \dots, f_{r})$. 
This leads to an upper bound on the number of rational points of $X$, improving the general
upper bound if every polynomial among $(f_{1}, \dots, f_{r})$ is relatively irreducible, 
but at least one is not absolutely irreducible (Proposition \ref{coarseci}).
In Section \ref{sec_tu} we construct a family of 
affine algebraic sets over ${\mathbb F}_q$ (the \emph{tubular sets}) reaching the general upper bound 
given in Section \ref{sec_b}. The corresponding projective family has also a large number of points
but does not reach the the general upper bound (Theorem \ref{nbptstub}).
It is worthwhile to precise that our results generalize and improve those previously obtained in the  
case of hypersurfaces defined over ${\mathbb F}_q$,
for which the best bounds are given in \cite{kalipe} in the affine case, and in 
\cite{serr} and \cite{sore2} in the projective case. Also note that some of our methods can be seen as 
similar, although in a more explicit and precise way, to the general approach of Heath-Brown in \cite[Th. 3]{khisbraun}.

\section{The cumulative degree}
\label{sec_cd}

Let $k$ be a field and $K$ the algebraic closure of $k$.  We are interested in the solutions in the affine space $k^{n}$ of a system
\begin{equation}
\label{systemaff}
f_{i}(T_{1}, \dots, T_{n}) = 0 \qquad (1 \leq i \leq r)
\end{equation}
with $f_{i} \in K[T_{1}, \dots, T_{n}]$, and $r \leq n$.
We are also concerned about solutions in the projective space $\PP^{n}(k)$ of a system
\begin{equation}
\label{systemproj}
f_{i}(T_{0}, \dots, T_{n}) = 0 \qquad (1 \leq i \leq r)
\end{equation}
with homogeneous polynomials $f_{i} \in K[T_{0}, \dots, T_{n}]$. These systems define 
respectively a $K$-algebraic subset $X$ in the affine space $\Aff^{n} = K^{n}$ and in the  
projective space $\PP^{n} = \PP^{n}(K)$. If $\mfa$ is an ideal of $K[T_{1}, \dots, T_{n}]$, 
the subset of zeros of $\mfa$ is denoted by $V(\mfa)$. Hence, $X = V(\mfa)$ where $\mfa$ is 
the ideal generated by $f_{1}, \dots, f_{r}$. If $S$ is a subset of $\Aff^{n}$ or $\PP^{n}$, 
the \emph{ideal of $S$}, denoted by $I(S)$, is the radical ideal of polynomials 
vanishing on $S$. Hence, $I(X)$ is the radical $\mfr(\mfa)$ of $\mfa$.

Let $Z_{1}, \dots, Z_{t}$ be the irreducible components of $X$, in such a way that
$$X = Z_{1} \cup \dots \cup Z_{t}.$$
We put $m = \dim X = \displaystyle{\max_{1 \leq i \leq t}} \dim Z_{i}$. 
Then $m \geq n - r$ since, by the Generalized Principal Ideal Theorem \cite[Ch. VIII, \S 3, Prop. 4]{BkiAC89}:
$$
\displaystyle{\min_{1 \leq i \leq t}} \dim Z_{i} \geq n - r.
$$
We denote by $\deg X$ the \emph{(usual) degree} of $X$, 
for which we refer to Fulton \cite{fult2}, 
Harris \cite{harr}, and Hartshorne \cite{hart}. 
Recall that if $X$ is of dimension $m$, then $\deg X$ is equal to 
$\card{X \cap L}$ for almost all linear varieties $L$ of complementary dimension $n - m$.  
For $0 \leq l \leq m$, put
$$
\cumdeg_{l} X = \sum_{\dim Z_{i} = l, \, 1 \leq i \leq t} \deg Z_{i}.
$$
The \emph{cumulative degree} of $X$ (Heintz \cite{hein}, Burgisser \cite{buclsh}) is
$$
\cumdeg X = \sum_{l = 0}^{m} \cumdeg_{l} X = \sum_{i = 1}^{t} \deg Z_{i}.
$$
Since \cite[Ex. 2.5.2(b)]{fult2} or \cite[Prop. 7.6(b)]{hart}:
$$
\deg(X) = \cumdeg_{m}(X),
$$
we have
$$
\deg(X) \leq \cumdeg(X),
$$
with equality if and only if $X$ is equidimensional of dimension $m$ 
(\idest every irreducible component of $X$ has dimension $m$).

There are many ways to state B\'ezout's Theorem. The more general one is the \emph{Main Theorem} and 
its \emph{refined version}, see Fulton \cite[Th. 12.3 and Ex. 12.3.1]{fult2} 
and Vogel \cite[Th. 2.1 and Cor. 2.26]{voge}. We use here three variants of B\'ezout's Theorem 
which are the more appropriate for our purposes, namely, Theorems \ref{bezout1}, \ref{fulton}, and \ref{bezout2}. 
Although they can be undoubtedly deduced from the general theory, we give in each case specific references for these statements.

\begin{theorem}[B\'ezout's Theorem, cumulative degree]
\label{bezout1}
let $Z$ be an algebraic subset, 
and $H_{1}, \dots, H_{r}$ a sequence of hypersurfaces in $\Aff^{n}$ or $\PP^{n}$. Then
$$
\cumdeg(Z \cap H_{1} \cap \ldots \cap H_{r}) \leq \cumdeg(Z) \prod_{i = 1}^{r}\deg (H_{i}).
$$
\end{theorem}

\begin{proof}
See Heintz \cite[Th. 1]{hein}, Burgisser \& al. \cite[Prop. 8.28]{buclsh}.
\end{proof}

Theorem \ref{bezout1} shows that if $X$ is given by \eqref{systemaff} or \eqref{systemproj}, 
by taking for $Z$ the whole space, and for $H_{i}$ the hypersurface $V(f_{i})$, then
\begin{equation}
\label{bezouthyp}
\cumdeg(X) \leq \prod_{i = 1}^{r}\deg (f_{i}).
\end{equation}

\begin{example}
Consider the couple of polynomials
$$
f_{1}(T_1,T_2) = T_2(T_2 - 1), \quad f_{2}(T_1,T_2) = T_1 T_2,
$$
and let $X = V(f_{1}, f_{2})$. Then $X = Z_{1} \cup Z_{2}$, where
$Z_{1}$ is the line $T_2 = 0$ and $Z_{2}$ is the point $(0, 1)$. Hence,
$$\deg X = 1, \quad \cumdeg X = 2, \quad (\deg f_{1})(\deg f_{2}) = 4.$$
\end{example}

\section{Bounds for algebraic sets}\label{sec_b}

\subsection{General case}

Here $k = \Fq$ is the field with $q$ elements, and $K = \bar{\FF}_{q}$.

\begin{theorem}
\label{genbound}
Let $X$ be a $K$-algebraic set of dimension $m$ in $\Aff^{n}$  (resp. $\PP^{n}$). If $X$ is affine, then
$$
\card{X \cap \Fq^{n}} \leq \sum_{l = 0}^{m} \cumdeg_{l}(X) q^{l} \leq \cumdeg(X) q^{m}.
$$
If $X$ is projective, then
$$
\card{X \cap \PP^{n}(\Fq)} \leq \sum_{l = 0}^{m} \cumdeg_{l}(X) \pi_{l} \leq \cumdeg(X) \pi_{m},
$$
where we have put $\pi_{n} = \card{\PP^{n}(\Fq)} = q^{n} + \dots + 1$ for $n \geq 0$.
\end{theorem}

With the help of \eqref{bezouthyp} we get

\begin{corollary}
\label{cibound}
$X$ be a $K$-algebraic set of dimension $m$ in $\Aff^{n}$  (resp. $\PP^{n}$), which is the zero set of 
a family of polynomials $(f_{1}, \dots, f_{r})$. Let $d_{i} = \deg f_{i}$. Then
$$
\card{X \cap \Fq^{n}} \leq d_{1} \dots d_{r} \ q^{m}, \, \text{resp.} \
\card{X \cap \PP^{n}(\Fq)} \leq d_{1} \dots d_{r} \, \pi_{m}.
$$
\end{corollary}

If we are only interested in the points of $X \cap \Fq^{n}$, 
one can replace in Corollary \ref{cibound} the polynomials $f_{i}$ by their reduction modulo the ideal generated 
by $T_{1}^{q} - T_{1}, \dots, T_{n}^{q} - T_{n}$, in such a way that $\deg f_{i} \leq n(q - 1)$.

If $X$ is not defined over $\Fq$, the bound of Theorem \ref{genbound} can be rough: see Corollary \ref{shortbound}.

The following proposition is a particular case of Theorem \ref{genbound}, but implies immediately this theorem. 
We define a $k$-variety as a $k$-irreducible algebraic set.

\begin{proposition}
\label{irrbound}
If $X$ is an affine (resp. projective) $K$-subvariety in $\PP^{n}$ of dimension 
$m$ in $\Aff^{n}$, resp. $\PP^{n}$, then
$$
\card{X \cap \Fq^{n}} \leq (\deg X) q^{m}, \quad
\text{resp. } \card{X \cap \PP^{n}(\Fq)} \leq (\deg X) \pi_{m}.
$$
\end{proposition}

\begin{proof}[Proof of Theorem \ref{genbound}] We assume that $X$ is affine, 
the argument being similar if $X$ is projective. Let $Z_{1}, \dots, Z_{t}$ be the 
irreducible components of $X$. Since
$$
\card{X \cap \Fq^{n}} \leq \sum_{l = 0}^{m} \sum_{\dim Z_{i} = l} \card{Z_{i} \cap \Fq^{n}},
$$
we have, by Prop. \ref{irrbound}
$$
\card{X \cap \Fq^{n}} \leq \sum_{l = 0}^{m} \sum_{\dim Z_{i} = l} \deg(Z_{i}) q^{l} =
\sum_{l = 0}^{m} \cumdeg_{l}(X) q^{l}.
$$
\end{proof}

It remains to prove Proposition \ref{irrbound}. If $X$ is defined over $\Fq$, 
these results are proved in \cite[Prop. 2.3]{lach2} and \cite[Prop. 12.1]{ghla} 
(the hypothesis of equidimensionality must be added in the statements), 
and also in \cite[Lemma 3.1]{frhaja}. We provide here a complete proof for reader's convenience.

A $K$-subvariety $X \subset \Aff^{n}$  (resp. $\PP^{n}$) is called \emph{nondegenerate} 
if it does not lie in any hyperplane, \idest if the $K$-linear subvariety generated 
by $X$ is equal to the whole space. We show in the next result that 
enumeration problems for the number of points of general $K$-subvarieties can 
be reduced to nondegenerate ones, whether they are rational over $\Fq$ or not.

\begin{lemma}
\label{nondegenerate}
Let $X$ be a $K$-subvariety in $\PP^{n}$, which is not a linear variety. 
There is a projection $\PP^{n} \longrightarrow \PP^{r}$ inducing an isomorphism of 
$X$ onto a nondegenerate $K$-subvariety $X' \subset \PP^{r}$, and
$$
\deg X' = \deg X, \quad
\card{X \cap \PP^{n}(\Fq)} \leq \card{X' \cap \PP^{r}(\Fq)}.
$$
The same result holds for a $K$-subvariety in an affine space.
\end{lemma}

\begin{proof}
If $X$ is nondegenerate there is nothing to prove. Suppose that $X$ is included in 
a hyperplane $H \subset \PP^{n}$. We can, and will, assume that $H$ is given by
$$
T_{n} = l(T_{0}, \dots, T_{n - 1}),
$$
with a linear form $l$ with coefficients in $K$. Let
$$\varphi(T_{0} : \ldots : T_{n}) = (T_{0} : \ldots : T_{n - 1})$$
be the projection from $(0 : \ldots : 0 : 1)$. The inverse map $\psi : \PP^{n - 1} \longrightarrow H$ of the 
restriction of $\varphi$ to $H$ is
$$
\psi : (T_{0} : \ldots : T_{n - 1}) \mapsto
(T_{0} : \ldots : T_{n - 1} : l(T_{0}, \ldots, T_{n - 1})).
$$
If $X$ is defined by
$$
f_{i}(T_{0}, \dots, T_{n}) = 0 \quad (1 \leq i \leq r),
$$
then $X' \subset \PP^{n - 1}$ is defined by
$$
f_{i}(T_{1}, \ldots, T_{n - 1}, l(T_{1}, \ldots, T_{n - 1}))  = 0 \quad (1 \leq i \leq r),
$$
and $\varphi$ defines a $K$-isomorphism from $X$ onto $X'$. 
Now $\deg X' = \deg X$ because $\varphi$ is a projection \cite[Ex. 18.16, p. 234]{harr}, 
and since $\varphi$ maps $\PP^{n}(\Fq)$ onto $\PP^{n - 1}(\Fq)$, 
we have $\card{X \cap \PP^{n}(\Fq)} \leq \card{X' \cap \PP^{n - 1}(\Fq)}$. 
If $X'$ is nondegenerate, we are done. In the opposite, we repeat the construction, 
and this comes to an end by infinite descent. The proof is the same for 
subvarieties in affine spaces.
\end{proof}

\begin{remark}
Since, \textit{a priori}, $\psi$ does not map  $\PP^{n - 1}(\Fq)$ into $\PP^{n}(\Fq)$, 
equality has no reason for being in the proposition.
\end{remark}

\begin{lemma}
\label{hyper}
If $X$ is a nondegenerate subvariety of dimension $m$ in $\Aff^{n}$ (resp. in $\PP^{n}$), 
if $H$ is a hyperplane, and if $X \cap H \neq \emptyset$, then $X$ and $H$ \emph{intersect properly},  
that is, $X \cap H$ is equidimensional of dimension $m - 1$. Moreover, $\deg(X \cap H) = \deg(X)$. 
\end{lemma}

\begin{proof}
See Hartshorne \cite[Ex. I.1.8]{hart} 
or Harris \cite[Ex. 11.6]{harr}. About the degree, see \cite[Prop. I.7.6(d)]{hart}.
\end{proof}

\begin{proof}[Proof of Proposition \ref{irrbound}]
By induction on $m$. If $m = 0$ then
$$\card{X(\Fq)} \leq \card{X(\Fqb)} = \deg X.$$
The proposition is obvious if $X$ is a linear variety. Otherwise, we can 
assume by Lemma \ref{nondegenerate} that $X$ is nondegenerate. 
If $H$ is a hyperplane, then $X \cap H$ is equidimensional of dimension $m - 1$ and of degree $d$  
by Lemma \ref{hyper}. We denote by $Z_{1} \dots, Z_{t}$ the irreducible components of $X \cap H$.

a) The proof is straightforward if $X$ is affine. By the induction hypothesis,
$$
\card{Z_{i} \cap \Fq^{n}} \leq \deg(Z_{i}) q^{m - 1}.
$$
and
$$
\card{X \cap H \cap \Fq^{n}} \leq  \sum_{i = 1}^{t} \deg(Z_{i}) q^{m - 1} = \deg(X) q^{m - 1}.
$$
Denote now by $H_{c}$ the hyperplane $x_{1} = c$. If $c \in \Fq$, we deduce from the preceding inequality
$$
\card{X \cap H_{c} \cap \Fq^{n}} \leq \deg(X) q^{m - 1}.
$$
Since
$$
X = \bigcup_{c \in \Fq} [X \cap H_{c}],
$$
we get the result is proved if $X$ is affine.\\
b) Assume $X$ projective. Let $\GG_{n - 1}$ be the variety of hyperplanes 
in $\PP^{n}$, and $T$ the \emph{incidence correspondence} \cite[\S \ 6.12]{harr} defined as
$$
T = \set{(x, H) \in X \times \GG_{n - 1}}{x \in X \cap H}.
$$
Although $T$ is not defined over $\Fq$, we define, by abuse of notation
$$
T(\Fq) = \set{(x, H) \in (X \cap \PP^{n}(\Fq)) \times \GG_{n - 1}(\Fq)}
{x \in X \cap H \cap \PP^{n}(\Fq)}.
$$
We get a diagram of sets with two projections
$$
\xymatrix{
& T(\Fq) \ar[dl]_{p_{1}} \ar[dr]^{p_{2}}  \\
X \cap \PP^{n}(\Fq) & & \GG_{n - 1}(\Fq)}
$$
If $x \in X \cap \PP^{n}(\Fq)$ then $p_{1}^{-1}(x)$ is in bijection with the set 
of hyperplanes $H$ in $\GG_{n - 1}(\Fq)$ containing $x$; hence $\card{p_{1}^{-1}(x)} = \pi_{n - 1}$ and
\begin{equation}
\label{T1}
\card{T(\Fq)} = \pi_{n - 1} \card{X \cap \PP^{n}(\Fq) }.
\end{equation}
On the other hand, if $H \in \GG_{n - 1}(\Fq)$, then $p_{2}^{-1}(H)$ is in 
bijection with the intersection $X \cap H \cap \PP^{n}(\Fq)$, hence
\begin{equation}
\label{T2}
\card{T(\Fq)} = \sum_{H} \card{X \cap H \cap \PP^{n}(\Fq)},
\end{equation}
where $H$ runs over the whole of $\GG_{n - 1}(\Fq)$. By the induction hypothesis, we see as in the affine case that
$$\card{X \cap H \cap \PP^{n}(\Fq)} \leq \deg(X) \pi_{m - 1}.$$
Since $\card{\GG_{n - 1}(\Fq)} = q \pi_{n - 1}$, we get from \eqref{T2} and the preceding inequality
$$
\card{T(\Fq)} \leq \deg(X) q \pi_{n - 1} \pi_{m - 1},
$$
we deduce from \eqref{T1} that $\card{X \cap \PP^{n}(\Fq)} \leq d q \pi_{m - 1} < d \pi_{m}$, 
and the result is proved if $X$ is projective.
\end{proof}

\subsection{Irrational subsets}

Here $k$ is a field and $K = \bar{k}$.

One can improve the preceding results for the number of $k$-rational points of an \emph{irrational} algebraic subset, 
that is, a Zariski closed subset of $\Aff^{n}$ or $\PP^{n}$ which is not defined over $k$. Let $k'$ 
be a Galois extension of $k$, with $G = \Gal(k'/k)$ and $s = [k':k]$. If $\sigma \in G$, we put
$$
f^{\sigma}(T_{1},\ldots, T_{n}) = \sum_{\alpha} c_{\alpha}^{\sigma} T^{\alpha}
\quad \text{if} \quad
f(T_{1},\ldots, T_{n}) = \sum_{\alpha} c_{\alpha} X^{\alpha}Ê\in k'[T_{1},\ldots, T_{n}],
$$
in such a way that $[f(x)]^{\sigma} = f^{\sigma}(x^{\sigma})$ for every $x \in \Aff^{n}$. 
Then $f \mapsto f^{\sigma}$ is an automorphism of the algebra $k'[X_{1},\ldots, X_{n}]$. 
If $\mfa$ is an ideal of $k'[T_{1},\ldots, T_{n}]$, we define
$$
\mfb = \sum_{\sigma \in G} \mfa^{\sigma}.
$$
Since $\mfb^{\sigma} = \mfb$ for every $\sigma \in G$, there is, by 
Galois descent \cite[Ch. V, \S 10, Prop. 6]{BkiAlg2En}, a $k$-structure on $\mfb$, 
that is, an ideal $\mfb_{k} \subset \mfb$ in $k[T_{1},\ldots, T_{n}]$ such that the homomorphism
$$
\begin{CD}
\mfb_{k} \otimes_{k} k' & @>>> & \mfb
\end{CD}
$$
is an isomorphism. Then
$$
\mfb = \mfb_{k}B, \quad \mfb_{k} = \mfb \cap A,
$$
and $\mfb$ is the \emph{smallest ideal of $B$ containing $\mfa$ defined over $k$}.

\begin{remark}
A family of generators of $\mfb_{k}$ can be deduced from a set of 
generators $\{f_{1}, \dots, f_{r}\}$ of $\mfa$ in the following way. The family
$$f_{1}, \dots, f_{r}, \dots, f_{1}^{\sigma}, \dots, f_{r}^{\sigma}, \dots$$
generates $\mfb$. Let $(\xi_{1}, \dots \xi_{s})$ a basis of $k'$ over $k$. 
There is a unique family $(g_{ij})$ of polynomials in $k[T_{1},\ldots, T_{n}]$ such that
\begin{equation}
\label{basech1}
f_{i}^{\sigma}(T_{1},\ldots, T_{n}) =
\sum_{j = 1}^{s} \xi_{j}^{\sigma} g_{ij}(T_{1},\ldots, T_{n}), \quad
1 \leq i \leq r, \sigma \in G.
\end{equation}
Let $(\eta_{1}, \dots \eta_{s})$ be the basis of $k'$ dual to $(\xi_{1}, \dots \xi_{s})$, in such a way that
$$
\Tr_{k'/k}(\xi_{i} \eta_{j}) =
\sum_{\sigma \in G} \xi_{i}^{\sigma} \eta_{j}^{\sigma} = \delta_{ij} \quad
1 \leq i,j \leq s.
$$
Then
\begin{equation}
\label{basech2}
g_{ij}(T_{1},\ldots, T_{n}) =
\sum_{\sigma \in G} \eta_{j}^{\sigma} f_{i}^{\sigma}(T_{1},\ldots, T_{n}), \quad
1 \leq i,j \leq s.
\end{equation}
The two formulas \eqref{basech1} and \eqref{basech2} show that $(g_{ij})$ is a family of generators 
of $\mfb$ in $k[T_{1},\ldots, T_{n}]$, that is, a family of generators of $\mfb_{k}$. 
\end{remark}

Let $X = V(\mfa)$ and $X^{\sigma} = V(\mfa^{\sigma})$. We put
$$
Q_{k'/k}(X) = V(\mfb_{k}) = \bigcap_{\sigma \in G}ÊX^{\sigma}.
$$
The algebraic subset $Q_{k'/k}(X)$ depends only on $X$ and not of $\mfa$ : it is 
the \emph{greatest $k$-closed subset in $X$}. If $\mfj = \mfr(\mfb)$ is the 
ideal of $Q_{k'/k}(X)$, then $\mfj_{k} = \mfj \cap A$ is a radical ideal, 
and $\mfj = \mfj_{k}B$ \cite[p. 22--24]{borel}. Moreover $\codim Q_{k'/k}(X) \leq s \codim X$.

\begin{lemma}
\label{rational}
Let $X$ be an algebraic subset of $\Aff^{n}$, defined over a Galois extension $k'$ of $k$, and $Y = Q_{k'/k}(X)$. Then
$$Y(k) = X \cap k^{n}.$$ 
\end{lemma}

\begin{proof}
If $x \in X \cap k^{n}$, then $f(x) = 0$ for every $f \in \mfa$ and every $\sigma \in G$, 
hence, $f(x) = 0$ for every $f \in \mfb$, hence, $f \in Y(k^{n})$. The reciprocal is immediate.
\end{proof}

\begin{theorem}[B\'ezout's Theorem, general version]
\label{fulton}
Let there be given $r$ equi\-di\-men\-sio\-nal subvarieties $X_{1}, \dots, X_{r}$ of $\PP^{n}$. Then 
$$
\deg (X_{1} \cap \ldots \cap X_{r}) \leq \deg (X_{1}) \dots \deg (X_{r}).$$
\end{theorem}

\begin{proof}
Fulton \cite[8.4.6 and Ex. 12.3.1]{fult2}, Vogel \cite[Cor. 2.26]{voge}.
\end{proof}

Recall that the set of fields of definition of an algebraic subset of $\Aff^{n}$ or $\PP^{n}$ has a 
smallest element \cite[4.8.11]{EGA42}.

\begin{proposition}\label{ineqdeg}
Let $X$ be an absolutely irreducible algebraic subset of $\Aff^{n}$, and $k'$ the smallest field of 
definition of $X$. Let $Y = Q_{k'/k}(X)$, and assume that $s = [k': k] > 1$. Then
$$\dim Y \leq \dim X - 1, \quad
\deg Y \leq (\deg X)^{s}.$$
\end{proposition}

\begin{proof}
Recall that $Y \subset X$. If $\dim Y = \dim X$, then $Y$ contains an irreducible component of $X$. 
Since $X$ is irreducible, we find that $X = Y$ and $X$ is defined over $k$, contrary to the hypotheses. 
Since $\deg X^{\sigma} = \deg X$, we get the last inequality by Theorem \ref{fulton}.
\end{proof}

As a consequence of Lemma \ref{rational} and Proposition \ref{ineqdeg}, the following holds:
\begin{corollary}
\label{shortbound}
Assume $k = \Fq$. Let $X$ be an absolutely irreducible algebraic subset of $\Aff^{n}$, of dimension $m$, and $k'$ 
the smallest field of definition of $X$. Let $Y = Q_{k'/k}(X)$, and assume $s = [k': k] > 1$. Then
$$
\card{X \cap \FF_{q}^{n}} \leq (\deg Y) \ q^{m - 1} \leq (\deg X)^{s} \ q^{m - 1}.
$$
\end{corollary}

This bound is better than the usual one if $q \geq (\deg X)^{s - 1}$. It is worthwile to specify that in 
this corollary, the integer $s$ depends on $q$. Analogous results hold for subsets of $\PP^{n}$ 
(substitute $\pi_{m - 1}$ to $q^{m - 1}$ in Corollary \ref{shortbound}).

\begin{example}
\label{ExCubic}
Let $k$ be a field of characteristic $\neq 3$. We assume that $2$ is not a square in $k$, and that $1$ 
has three cube roots in $k$. Let $k' = k(\sqrt{2})$. Define
$$
f(T_1,T_2,T_3) = T_1^{3} - T_2^{3} + \sqrt{2} (T_3^{3} - T_2^{3}),
$$
and let $C$ the plane projective cubic defined over $k'$ with equation $f = 0$. One checks that $C$ is nonsingular, 
hence, $C$ is absolutely irreducible. If
$$
g_{1}(T_1,T_2,T_3) = T_1^{3} - T_2^{3}, \quad g_{2}(T_1,T_2,T_3) = T_3^{3} - T_2^{3},
$$
then
$$
f(T_1,T_2,T_3) = g_{1}(T_1,T_2,T_3)+ \sqrt{2} g_{2}(T_1,T_2,T_3),
$$
and the algebraic subset $Y$ of $\PP^{2}$ defined by
$$
g_{1}(T_1,T_2,T_3) = 0, \quad g_{2}(T_1,T_2,T_3) = 0
$$
is of dimension $0$. Here, $s = 2$ and there are $(\deg C)^{2} = 9$ points in $Y(k) = C \cap \PP^{2}(k)$, 
namely the points  $(1: \eta : \zeta)$, 
where $\eta$ and $\zeta$ runs over the three cubic roots of $1$.
\end{example}

\begin{remark}
\label{chevwarn}
We are merely concerned here by upper bounds on the number of points of algebraic sets. 
But it is worthwile to recall the lower bounds obtained with the help of the Chevalley-Warning Theorem. 
Let $(f_{1}, \dots, f_{r})$ be a sequence of polynomials in $k[T_{1}, \dots, T_{n}]$, 
of respective degrees $d_{1}, \dots, d_{r}$, and write $d = d_{1} + \dots + d_{r}$. 
Let $X$ be the set of zeros of these polynomials in $\Aff^{n}$. Then Warning \cite{warning} 
proved that if $d \leq n$ and $X(k) \neq \emptyset$, then
$$\card{X(k)} \geq q^{n - d}.$$
See \cite{khisbraun} for a discussion and improvements of this result.
\end{remark}

\section{Relatively irreducible sets}
\label{sec_rci}

Here $k$ is a field and $K = \bar{k}$.

\subsection{Decomposition in absolute components}

\begin{theorem}
\label{DecIdeal}
Let $X$ be a $k$-irreducible subset in $\Aff^{n}$ or $\PP^{n}$, $k(X)$ the field of 
rational functions on 
$X$, and $k'$ the (relative) algebraic closure of $k$ in $k(X)$. Assume that $k'$ is 
a Galois extension of 
$k$ and let $G = \Gal(k'/k)$. Let $\mathsf{Z}$ be the set of absolutely 
irreducible components of $X$ and $Z \in \mathsf{Z}$.
\begin{enumerate}[(i)]
\item
\label{DCO1}
The smallest field of definition of every element of $\mathsf{Z}$ is equal to $k'$, 
and $\mathsf{Z}$ has $[k' : k]$ elements.
\item 
\label{conjug}
Let $\mathfrak{p}$ the ideal of $X$ in $A = k[T_{1}, \dots, T_{n}]$, and $\mathfrak{P}$ the ideal of $Z$ in $B = k'[T_{1}, \dots, T_{n}]$. Then
$$
\mathfrak{p}B = \bigcap_{\sigma \in G} \mathfrak{P}^{\sigma}.
$$
\end{enumerate}
\end{theorem}

\begin{proof}
The statement \eqref{DCO1} is proved in EGA in a more general setting, see 
\cite[Prop. 4.5.10]{EGA42}. Let $\mathfrak{P}_{i}$ the ideal of $Z_{i}$ in $B = k'[T_{1}, \dots, T_{n}]$. Then
$$
X = \bigcup_{i = 1}^{s} Z_{i}, \qquad \mathfrak{p}B = \bigcap_{i = 1}^{s} \mathfrak{P}_{i}.
$$
Let $\Pi = \{\mathfrak{P}_{1}, \dots, \mathfrak{P}_{s} \}$. The group $G$ 
operates transitively on $\Pi$ 
\cite[Ch. V, \S 2, Th. 2]{BkiCAEn}, and even simply transitively on $\Pi$, 
since $\card{G} = \card{\Pi} = s$, 
according to \eqref{DCO1}. The result follows.
\end{proof}

\begin{theorem}
\label{DecComp}
Let $X$ be a $k$-irreducible subset in $\Aff^{n}$ or $\PP^{n}$, $k(X)$ the field of 
rational functions on 
$X$, and $k'$ the (relative) algebraic closure of $k$ in $k(X)$. Assume that $k'$ is 
a Galois extension of 
$k$ and let $G = \Gal(k'/k)$. Let $\mathsf{Z}$ be the set of absolutely 
irreducible components of $X$. Then:
\begin{enumerate}[(i)]
\item
\label{DCO2}
$X$ is equidimensional.
\item
\label{DCO4}
$G$ operates simply transitively on $\mathsf{Z}$ : if we choose $Z \in \mathsf{Z}$, then
$$
X = \bigcup_{\sigma \in G} Z^{\sigma}.
$$
\item
\label{DCO9}
For any $Z\in \mathsf{Z}$ one has $\deg X = [k' : k] \deg Z$.
\end{enumerate}
\end{theorem}

\begin{proof}
The statement \eqref{DCO2} is proved in EGA in a more general setting, see 
\cite[Prop. 5.2.1]{EGA42}. 
The statement \eqref{DCO4} follows from Theorem \ref{DecIdeal}\eqref{conjug}.
By the results of section \ref{sec_cd}
$$\deg X = \sum_{i = 1}^{s} \deg Z_{i} \ ;$$
but by \eqref{DCO4} all the $Z_i$ of $\mathsf{Z}$ have the same degree,
this proves \eqref{DCO9}.
\end{proof}

\begin{remark}
Denote by $S_{k'/k}(Z)$ the Zariski $k$-closure of $Z$. Then Theorem \ref{DecComp}\eqref{DCO4} 
means that $X = S_{k'/k}(Z)$. With the help of Lemma \ref{rational}, we get
$$Q_{k'/k}(Z) \subset Z \subset S_{k'/k}(Z),$$
and the preceding results show that this is true for every Zariski closed subset 
$Z$ of $\Aff^{n}$ or $\PP^{n}$,provided that the smallest field $k'$ of definition of $Z$ is 
a Galois extension of $k$. One could say that $Q_{k'/k}(Z)$ is ``inscribed'' in $Z$, 
and that $S_{k'/k}(Z)$ is ``escribed'' to $Z$. Moreover,
$$Q_{k'/k}(Z)(k) = Z \cap k^{n} = S_{k'/k}(Z)(k).$$
\end{remark}

\begin{example}
In Example \ref{ExCubic}, denote by $\sigma$ the automorphism of $k'$ such 
that $(\sqrt{2})^{\sigma} = - \sqrt{2}$. Let
$$
h(T_1,T_2,T_3)=
f(T_1,T_2,T_3) f^{\sigma}(T_1,T_2,T_3) = (T_1^{3} - T_2^{3})^{2} - 2 (T_3^{3} - T_2^{3})^{2}
$$
Then $h$ is irreducible in $k[T_1,T_2,T_3]$, and the plane curve $D$ defined over $k$ 
with equation $h = 0$ is $k$-irreducible. Then
$$
\sqrt{2} = \frac{T_1^{3} - T_2^{3}}{T_3^{3} - T_2^{3}}
$$
and $k'$ is the algebraic closure of $k$ in $k(D)$. The absolutely irreducible 
components of $D$ are the 
curves $C$ and $C^{\sigma}$, defined over $k'$. The $9$ points of $C(k)$ are 
the points of $C \cap C^{\sigma}$; 
they are the singular points of $D$.
\end{example}

\begin{example}
\cite[Ex. 2.6]{Kollar}.
Let $k = \FF_{q}$, $k' = \FF_{q^{n}}$, $G = \Gal(k'/k)$ and let $\alpha$ be a generator of $k'/k$. Define
$$g(T_{0},T_{1},\dots T_{n}) = T_{1} + \alpha T_{2} + \alpha^{n - 1} T_{n},$$
and put
$$
f = \prod_{\sigma \in G} g^{\sigma} .
$$
The hypersurface $X$ with equation $f = 0$ is defined over $k$ and $k$-irreducible, 
of degree $n$, 
and $\card{X(k)} = 1$, with one point $(1 : 0 : \dots : 0) \in \PP^{n}$. 
The algebraic closure of $k$ in $K(X)$ 
is equal to $k'$, and if $Z$ is the hyperplane with equation $g = 0$, 
the hypersurface $X$ is the union of the hyperplanes $Z^{\sigma}$.
\end{example}

\begin{corollary}
\label{relirrbound}
Assume $k = \FF_{q}$. Let $X$ be a $k$-irreducible subset of dimension $m$ in $\PP^{n}$, and $k'$ 
the algebraic closure of $k$ in $k(X)$. Let $s' = [k':k]$.
\begin{enumerate}[(i)]
\item
We have
$$
\card{X(\FF_{q})} \leq \frac{\deg X}{s'} \ \pi_{m}.
$$
\item
If $s' > 1$, then
$$
\card{X(\FF_{q})} \leq \left( \frac{\deg X}{s'} \right)^{s} \ \pi_{m - 1}.
$$
\end{enumerate}
\end{corollary}

\begin{proof}
Observe that $X(\FF_{q}) = Z \cap \FF_{q}^{n}$, where $Z$ is an absolutely irreducible 
component of $X$, and apply 
successively Theorem \ref{genbound}  and Corollary \ref{shortbound} to $Z$.
\end{proof}

\begin{example}
\label{ExCubic2}
In Example \ref{ExCubic}, denote by $\sigma$ the automorphism of $k'$ such that $(\sqrt{2})^{\sigma} = - \sqrt{2}$. Let
$$
h(X,Y,Z) =
f(X,Y,Z) f^{\sigma}(X,Y,Z) = (X^{3} - Y^{3})^{2} - 2 (Z^{3} - Y^{3})^{2}
$$
Then $h$ is irreducible in $k[X,Y,Z]$, and the plane curve $D$ defined over $k$ with equation $h = 0$ is $k$-irreducible. Then
$$
\sqrt{2} = \frac{X^{3} - Y^{3}}{Z^{3} - Y^{3}}
$$
and $k'$ is the algebraic closure of $k$ in $k(D)$, with $s = [k':k] = 2$. The absolutely 
irreducible components of $D$ are the curves $C$ and $C^{\sigma}$, defined over $k'$. Indeed, 
the points of $D(k)$ are the points of $C \cap C^{\sigma}$, and they are the singular points of $D$. 
If $k = \FF_{q}$, then the inequality of Corollary \ref{relirrbound}(ii) is an equality:
$$
\card{D(\FF_{q})} = \left( \frac{\deg D}{s} \right)^{s} = 9.
$$

\end{example}

\subsection{Normal and rational points}

We transpose here to finite fields a remark of Serre about relatively irreducible 
varieties defined in 
characteristic $0$ \cite[p. 20]{SerreGalois}.

Let $k$ be any field, and $X$ a $k$-irreducible subset of dimension $m$ in $\Aff^{n}$ or $\PP^{n}$ 
defined over $k$. Recall 
that $X$ is normal at the point $x \in X$ if the local ring 
$\mathcal{O}_{x} \subset k(X)$ is integrally closed. 
The algebraic subset $X$ is normal if it is normal at every point. 
Let $\mathfrak{m}_{x}$ be the maximal 
ideal of $\mathcal{O}_{x}$, and $\kappa(x) = \mathcal{O}_{x}/\mathfrak{m}_{x}$ 
the residual field. 
If $L$ is an extension of $k$, Any point $x \in X(L)$ defines an 
injective homomorphism $\kappa(x) \longrightarrow L$. 
If $X$ is normal at $x$, then $\kappa(x)$ contains the relative algebraic 
closure $k'$ of $k$ in $k(X)$. 
Therefore $k' \subset L$ if $x$ is a normal point of $X$ and $x \in X(L)$.

Assume now that $X$ is not absolutely irreducible. Then $[k':k] > 1$, 
and no point $x \in X(k)$ can be normal.
Since a nonsingular point is normal, the set of points which are not normal 
are contained in the algebraic set $\Sing X$ of singular points of $X$, 
which is of codimension $\geq 1$. Therefore $X(k) \subset \Sing X$, 
and actually $X(k) = (\Sing X)(k)$ since $\Sing X \subset X$.

Recall that if  $Z_{1}, \dots, Z_{s}$ are the irreducible components of $X$, then
$$
\Sing X = \bigcup_{i = 1}^{s} \Sing Z_{i} \cup \bigcup_{i \neq j} Z_{i} \cap Z_{j}.
$$
Since $\dim Z_{i} \cap Z_{j} \geq \dim X - r$, we have $\dim \Sing X \geq \dim X - r$.

\begin{proposition}
\label{singbound}
Let $X$ be a $k$-irreducible subset in $\Aff^{n}$ or $\PP^{n}$ defined over $k$, 
which is not absolutely irreducible. Then $X(k) = (\Sing X)(k)$. Moreover $X(k) = \emptyset$ if $X$ is normal. 
If $k = \Fq$, then
$$
\card{X(k)} \leq (\deg \Sing X) \ \pi_{m - 1}.
$$
\end{proposition}

\begin{example}
Let $X$ be a complete intersection in $\Aff^{n}$ or $\PP^{n}$. If $\dim \Sing X \leq \dim X - 2$, 
then $X$ is normal by Serre's criterion \cite[Th. 5.8.6]{EGA42}, and $X(k) = \emptyset$.
\end{example}

\begin{example}
\label{ExCubic3}
In Example \ref{ExCubic2}, we saw that the $k$-rational points of $D$ are exactly the $9$ singular points of $D$, 
hence, $D(k) = \Sing D$ in this case.
\end{example}

\section{Complete intersections}
\label{sec_ci}

\subsection{Notation}

Let $X$ be an algebraic subset of $\PP^{n}$ of dimension $m = n - r$.
\renewcommand{\labelitemi}{---}
\begin{itemize}
\item
$X$ is a \emph{set-theoretic (s.t.) complete intersection} if $X = V(\mfa)$, where $\mfa$ 
is generated by $r$ homogeneous polynomials in $K[T_{0}, \cdots, T_{n}]$.
\item
$X$ is an \emph{ideal-theoretic (i.t.) complete intersection} if moreover $\mfa$ is a radical ideal, 
that is, if the ideal of $X$ can be generated by $r$ homogeneous polynomials.
\end{itemize} 
One defines in the same way complete intersections in $\Aff^{n}$. An s.t. complete intersection 
is equi\-dimensional, by the Generalized Principal Ideal Theorem
(section \ref{sec_cd}). The cumulative and the usual degree coincide, and
$$
\cumdeg X = \deg X = \sum_{i = 1}^{t} \deg Z_{i},
$$
where $Z_{1}, \dots, Z_{t}$ are the irreducible components of $X$. 
Let $X$ be an i.t. complete intersection, $(f_{1}, \dots, f_{r})$ the ideal of $X$, and $d_{i} = \deg f_{i}$. 
Since the numerator of the \emph{Hilbert series} of $X$ equals \cite[Th. III-83]{eiha}
$$
(1 - T^{d_{1}}) (1 - T^{d_{2}}) \dots (1 - T^{d_{r}}),
$$
the family $(d_{1}, \dots, d_{r})$, which is called the \emph{multidegree} of $X$, 
depends only of $X$ and not of the system of generators $f_{1}, \dots, f_{r}$.

\begin{theorem}[B\'ezout's Theorem, complete intersections]
\label{bezout2}
Let $X$ be an i.t. complete intersection with multidegree $(d_{1}, \dots, d_{r})$. Then
$$\deg X = \prod_{i = 1}^{r} d_{i}.$$
\end{theorem}

\begin{proof}
Eisenbud-Harris \cite[Th. III-71]{eiha} (using schemes), Vogel \cite[\S 1.35]{voge} (using regular sequences).
\end{proof}

\subsection{Coarse decomposition of complete intersections}

We shall now describe a decomposition of complete intersections, explicitly constructed from the system of equations defining $X$, 
which can be used as a substitute for the decomposition in irreducible components.

We first give a lemma from an unpublished paper of 1994 by the first author, stated 
without proof in \cite{sore2} and \cite{roll1}.
If $g \in B$, the polynomial
$$
N_{k'/k}(g) = \prod_{\sigma \in G} g^{\sigma} \in A
$$
is called a \emph{norm polynomial}.

\begin{lemma}
\label{hypersurf}
Let $f \in A$ be a $k$-irreducible polynomial, and $k'$ the algebraic 
closure of $k$ in the field of fractions of $A/(f)$. 
Assume that $k'$ is a Galois extension of $k$ and let $G = \Gal(k'/k)$. 
Then there is an absolutely irreducible polynomial $g \in B$ such that 
$$
f = N_{k'/k}(g),
$$
and $\deg f = [k':k] \deg g$.
\end{lemma}

\begin{proof}
Let $X$ be the hypersurface in $\Aff^{n}$ defined by $f$. 
With the notation of Lemma \ref{DecIdeal} 
and its proof, $fB = \mathfrak{p}B$ is generated by $f$. 
In the same way, $\mathfrak{P}$ is a principal 
ideal by \eqref{DCO2}. Let $g$ be a generator of $\mathfrak{P}$. 
From Theorem \ref{DecIdeal} \eqref{conjug} we deduce that
$$
f = c \prod_{\sigma \in G} g^{\sigma},
$$
with $c \in B$. But $g = \prod g^{\sigma}$ is in $A$, since it is invariant 
under $G$, and consequently $c \in A$. 
But $c \in k^{\times}$, since $f$ is $k$-irreducible. 
If we choose now an element $\lambda \in k'$ of norm $1/c$, 
and substitute $\lambda g$ to $g$, we get the result.
\end{proof}

Let $f_{1}, \dots f_{r}$ be a sequence of $k$-irreducible homogeneous polynomials in $A = k[T_{0}, \dots, T_{n}]$, 
of respective degrees $d_{1}, \dots, d_{r}$, such that the algebraic subset of $\PP^{n}$
$$
X = V(f_{1}, \dots f_{r}) = \bigcap_{i = 1}^{r} H_{i}, \quad
\text{with} \quad H_{i} = V(f_{i}),
$$
is defined over $k$, of dimension $m = n - r$. In other words, $X$ is a set-theoretic complete intersection, 
hence, equidimensional. For $\ \leq i \leq r$, let $k_{i}$ be the algebraic closure of $k$ in the field $k(H_{i})$ of 
rational functions on $H_{i}$. Assume $k_{i}$ is a finite Galois extension of $k$, and put $[k_{i} : k] = s_{i}$. By Lemma \ref{hypersurf}, 
there is an absolutely irreducible homogeneous polynomial  $g_{i} \in k_{i}[T_{0}, \dots, T_{n}]$ of degree $\deg g_{i} = e_{i}$ such that 
$$
f_{i} = \prod_{\sigma \in \Gamma_{i}} \ g_{i}^{\sigma}, \qquad
e_{i} = \frac{d_{i}}{s_{i}} \, ,
$$
with $\Gamma_{i} = \Gal(k_{i}/k)$. The smallest field of definition of
$$
Y = V(g_{1}, \dots g_{r}) = \bigcap_{i = 1}^{r} G_{i}, \quad
\text{with} \quad G_{i} = V(g_{i}),
$$
is the composite extension $k_{0}$ of the fields $k_{i}$. Then $k_{0}$ is Galois, and
$$
s_{0} = [k_{0} : k] \leq s, \quad \text{with} \quad s = s_{1} \dots s_{r}.
$$
If $k = \Fq$, then $s_{0} = \lcm(s_{1}, \dots, s_{r})$.
 
\begin{proposition}
\label{coarseci}
Let $X$ be a set theoretical (s.t.) complete intersection. With the preceding notation and hypotheses:
\begin{enumerate}[(i)]
\item
\label{BCI1}
The subset $Y$ is a s.t. complete intersection of dimension $m$, and its smal\-lest field of definition is $k_{0}$.
\item
\label{BCI2}
Let $\mathsf{Z}$ be the set of absolutely irreducible components of $X$. There is a subset $\mathsf{Z}(0) \subset \mathsf{Z}$ such that
$$
Y = \bigcup_{Z \in \mathsf{Z}(0)} Z.
$$
\item
\label{BCI3}
We have
$$\deg Y \leq e_{1} \dots e_{r} = \frac{d_{1} \dots d_{r}}{s}.$$
\item
\label{BCI4}
We have
$$X(k) = Y \cap \PP^{n}(k).$$
\item
\label{BCI5}
If $k = \Fq$, then
$$
\card{X(k)} \leq  (\deg Y) \, \pi_{m}.
$$
\end{enumerate}
A similar result holds for algebraic subsets of $\Aff^{n}$.
\end{proposition}

\begin{proof}
By the Generalized Principal Ideal Theorem (section \ref{sec_cd}), $\dim Y \geq m$. 
But $Y \subset X$, hence $\dim Y = m$, and $Y$ is a s.t. complete intersection; 
this proves \eqref{BCI1}.  Since $Y$ is equidimensional of dimension $m$ and $Y \subset X$, we get \eqref{BCI2}.
One deduce \eqref{BCI3} from \eqref{bezouthyp}. The formula \eqref{BCI4} is immediate, and \eqref{BCI5} 
follows from \eqref{BCI4} and Theorem \ref{genbound}.
\end{proof}

\begin{remark}
From Corollary \ref{shortbound}, and by putting $W = Q_{k_{0}/k}(Y)$, we have
$$
\card{X(k)} \leq  (\deg W) \, \pi_{m - 1} \leq  (\deg Y)^{s_{0}} \, \pi_{m - 1}.
$$
\end{remark}

Since
$$H_{i} = \bigcup_{\sigma \in \Gamma_{i}} \ G_{i}^{\sigma},$$
we have, by putting $\Gamma = \Gamma_{1} \times \dots \times \Gamma_{r}$,
$$
X = 
\bigcap_{i = 1}^{r} \bigcup_{\sigma \in \Gamma_{i}} \ G_{i}^{\sigma} =
\bigcup_{a \in \Gamma} \bigcap_{i = 1}^{r} G_{i}^{a(i)},
$$
with $a = (a(1), \dots a(r)) \in \Gamma$. Define
$$
Y^{(a)} = \bigcap_{i = 1}^{r} G_{i}^{a(i)}.
$$

\begin{proposition}
\label{GenBasic}
The subset $Y^{(a)}$ is a s.t. complete intersection of dimension $m$, and its smallest field of definition is $k_{0}$. We have
\begin{equation}
\label{GBA1}
X = \bigcup_{a \in \Gamma} Y^{(a)}.
\end{equation}
For every $a \in \Gamma$, there is a subset $\mathsf{Z}(a) \subset \mathsf{Z}$ such that
$$
Y^{(a)} = \bigcup_{Z \in \mathsf{Z}(a)} Z. 
$$
\end{proposition}

The covering \eqref{GBA1} of $X$ is called the \emph{coarse decomposition} of $X$.

If $X$ is a $k$-irreducible complete intersection, it can be interesting to compare the coarse decomposition of $X$ with 
its decomposition in irreducible components detailed in Theorem \ref{DecComp}. This can be performed under suitable conditions; 
for instance, we have the following result.

\begin{proposition}
\label{IrrIntComp}
With the preceding notation, assume that $X$ is a $k$-irredu\-cible i.t. complete intersection over $k$, 
and that $Y$ is an i.t. complete intersection with field of definition $k_{0}$. Let $s = s_{1} \dots s_{r}$ as above, and $k'$ 
the smallest field of definition of the irreducible components of $X$. Then:
\begin{enumerate}[(i)]
\item
\label{IIC1}
The family $\{\mathsf{Z}(a)\}_{a \in \Gamma}$ is a partition of $\mathsf{Z}$ into $s$ subsets with $s'/s$ elements, and $k_{0} \subset k'$.
\item
\label{IIC2}
We have
$$\deg Y = e_{1} \dots e_{r} = \frac{\deg X}{s}.$$
\item
\label{IIC3}
The coarse decomposition of $X$ is irredundant.
\end{enumerate}
\end{proposition}

\begin{proof}
The subset $Y^{(a)}$ is defined over $k'$ by \eqref{GBA1}, hence, $k_{0} \subset k'$. If we choose $Z \in \mathsf{Z}$, then
$$
\deg Y^{(a)} = \card{\mathsf{Z}(a)} \deg Z,
$$
since all the elements of $\mathsf{Z}$ have the same degree by Theorem \ref{DecComp}. By Theorem \ref{bezout2}, we have
$$
\deg Y^{(a)} = e_{1} \dots e_{r}, \quad \deg X = d_{1} \dots d_{r} = s \deg Y^{(a)},
$$
hence,
$$
\deg X = s \deg Y^{(a)} = s \, \card{\mathsf{Z}(a)} \deg Z.
$$
and this proves \eqref{IIC2}. Since $\card{\mathsf{Z}} = s' = [k':k]$ and $\deg X = s' \deg Z$, we see that $s' = c s$. 
This proves \eqref{IIC1}, from which we see that no $Y^{(a)}$ can be dropped in the coarse decomposition, whence \eqref{IIC3}.
\end{proof}

If $Y$ is irreducible, \idest, $c = 1$, then the coarse decomposition is identical to 
the decomposition of $X$ in irreducible components. Otherwise the covering of $X$ by 
irreducible components refines the coarse one.

\begin{remark}
If $k = \Fq$ and $c > 1$, the bound 
$$
\card{X(k)} \leq  \frac{\deg X}{s'} \, \pi_{m}
$$
of Corollary \ref{relirrbound} is better than the bound
$$
\card{X(k)} \leq  \frac{\deg X}{s} \, \pi_{m}
$$
of Proposition \ref{coarseci}\eqref{BCI5}. If $c = 1$, they are identical.
\end{remark}

\begin{remark}
If the polynomials defining $X$ are $K$-irreducible, Proposition \ref{coarseci} does not bring anything new.
But if at least one of these polynomials is \emph{relatively irreducible}, that is, $\Fq$-irreducible but not
$K$-irreducible, then some of the $c_i$ are $\geq 2$, and the bound of
Proposition \ref{coarseci} is better than the one of Corollary \ref{cibound}.
\end{remark}

\section{Tubular sets}
\label{sec_tu}

We give in this section examples of algebraic sets with many points. 
The following construction generalizes a construction of Serre 
\cite{serr}, corresponding to the case of codimension $r = 1$. 
Take a sequence $d_{1}, \dots, d_{r}$ 
of integers $\geq 1$, and choose a family
$$
a_{i,j} \in \Fq, \quad i \in \{1, \dots r \}, \ j \in \{1, \dots, d_{i} \},
$$
where we assume, for every $i$, that $a_{i,j} \neq a_{i,k}$ if $j \neq k$. 
Note that this condition implies $d_{i} \leq q$ for $1 \leq i \leq r$. 
Denote by $\mfa_{i,j}$ the principal ideal 
$(T_{i} - a_{i,j} T_{0})$ of $\Fq[T_{0}, \dots, T_{n}]$, 
and by $H_{i,j} = V(\mfa_{i,j})$ the corresponding hyperplane of $\PP^{n}$. Let 
$$
D = \prod_{i = 1}^{r} \{1, \dots, d_{i} \}.
$$
If $J = (j(1), \dots, j(r)) \in D$, we put
$$
\mfp_{J} =
(T_{1} - a_{1, j(1)} T_{0}, T_{2} - a_{2, j(2)} T_{0}, \dots, T_{r} - a_{r, j(r)} T_{0}),
\quad Y_{J} = V(\mfp_{J}).
$$
We have
$$
\mfp_{J} = \sum_{i = 1}^{r} \mfa_{i,j(i)}, \qquad 
Y_{J} = \bigcap_{i = 1}^{r} H_{i,j(i)}.
$$
The projective linear variety $Y_{J}$ is defined by $r$ linearly independent forms, and hence, 
of dimension $n - r$. Observe that if $J$ and $K$ are in $D$ and if $J \neq K$, 
then $j(i) \neq k(i)$ for some $i$, hence $a_{i,j(i)} \neq a_{i,k(i)}$, and 
$H_{i,j(i)} \cap H_{i,k(i)}$ is the linear variety $X_{0}$ of dimension $n - r - 1$ with ideal
$$
\mfp_{0} = (T_{0}, T_{1}, T_{2}, \dots, T_{r}).
$$
This proves that $Y_{J} \cap Y_{K} = X_{0}$. The \emph{tubular set} $X$ defined 
by the family  $(a_{i,j})$ is the algebraic subset of $\PP^{n}$ which is the union 
of the varieties $Y_{J}$ (if $n = 3$, $r = 2$, the subsets $Y_{J}$ can be see 
as ``tubes'', whence the name ``tubular set''). If $\mfa$ is the ideal of $X$, then
\begin{equation}
\label{disjsum}
\mfa = \bigcap_{J \in D} \mfp_{J}, \quad \text{and} \quad
X = V(\mfa) = \bigcup_{J \in D} Y_{J}.
\end{equation}
The irreducible components of $X$ are the linear varieties $Y_{J}$, and $\dim X = n - r$. 
By distributivity of union over intersection,
$$
\mfa =
\bigcap_{j \in D} \left(\sum_{i = 1}^{r} \mfa_{i,j(i)} \right) =
\sum_{i = 1}^{r} \left(\bigcap_{j=1}^{d_{i}} \mfa_{i,j} \right).
$$
Similarly,
$$
X = 
\bigcup_{J \in D} \left( \bigcap_{i = 1}^{r} H_{i,j(i)} \right) =
\bigcap_{i = 1}^{r} \left( \bigcup_{j = 1}^{d_{i}} H_{i,j} \right)
$$
For $1 \leq i \leq r$, the ideal
$$
\mfa_{i} = \bigcap_{j=1}^{d_{i}} \mfa_{i,j}
$$
is principal, since $\mfa_{i} = (f_{i})$, where
$$
f_{i}(T_{0}, T_{1},\cdots,T_{n}) = \prod_{j=1}^{d_{i}} (T_{i} - a_{i,j} T_{0}),
$$
and $X_{i} = V(\mfa_{i})$ is an hypersurface of degree $d_{i}$: 
$$X_{i} = \bigcup_{j = 1}^{d_{i}} H_{i,j}.$$
Now
$\mfa = \mfa_{1} + \dots + \mfa_{r}$, and
$$
X = V(\mfa) = \bigcap_{i = 1}^{r} X_{i}.
$$
Then $X$ is an i.t. complete intersection, and
$$\deg X = \sum_{J \in D} \deg Y_{J} = \card{D} = \prod_{i = 1}^{r} d_{i}.$$

We write $\Aff^{n} = \PP^{n} \setminus H_{0}$, where $H_{0}$ is the hyperplane $T_{0} = 0$. 
The linear variety $Y_{J}$ is the disjoint union of the affine variety 
$Y_{J}' = Y_{J} \cap \Aff^{n}$ and of $X_{0} = Y_{J} \cap H_{0}$. 
The tubular set $X$ and its affine part $X' = X \cap \Aff^{n}$ are obtained as disjoint unions
\begin{equation}
\label{affineproj}
X = X_{0} \amalg X' \subset \PP^{n}, \quad
X'= \coprod_{J \in D} Y_{J}' \subset \Aff^{n}.
\end{equation}

The enumeration of points of a tubular set is as follows:

\begin{theorem}
\label{nbptstub}
Let $(d_{1}, \dots, d_{r}) \in \NN^{r}$ with $1 \leq d_i \leq q$ for any $i$ in $\{1,\cdots,r\}$. 
The tubular set $X \subset \PP^{n}$ defined above is an i.t. complete intersection, 
defined over $\Fq$, of dimension $m = n - r$, multidegree $(d_{1}, \dots, d_{r})$, 
and degree  $d = d_{1} \cdots d_{r} \leq q^{r}$.
\begin{enumerate}[(i)]
\item
\label{NPF1}
The affine algebraic subset $X'$ satisfies
$$
\card{X'(\Fq)} = d q^{m}.
$$
\item
\label{NPF2}
The projective algebraic subset $X$ satisfies
$$
\card{X(\Fq)} = d q^{m} + \pi_{m - 1} = d \pi_{m} - (d - 1) \pi_{m - 1}.
$$
\end{enumerate}
\end{theorem}

\begin{proof}
We only have to prove \eqref{NPF1} and \eqref{NPF2}. Apply \eqref{affineproj}: 
since $\card{Y_{J}'(\Fq)} = q^{m}$ for any  $J \in D$, we get \eqref{NPF1}, and we prove 
\eqref{NPF2} by observing that $\card{X_{0}(\Fq)} = \pi_{m - 1}$.
\end{proof}

Theorem \ref{nbptstub}\eqref{NPF1} shows that the bound of Corollary \ref{cibound} 
is attained in the affine case.
The projective case is different: we do not know any examples of i.t. 
complete intersections in $\PP^{n}$ reaching the bound of Corollary \ref{cibound}. Hence, we ask:

\begin{question}
What is the value of 
$$M_{m,d}(q) = \max_{X} \ \card{X(\Fq)},$$
when $X$ runs over projective i.t. complete intersections of dim. $m$ and degree $d$?
\end{question}
By the preceding, we know that
$$
d q^{n} + \pi_{n - 1} \leq M_{m,d}(q) \leq d q^{n} + d \pi_{n - 1}.
$$

In small codimension, the first author put forward in \cite[Conj. 12.2]{ghla} :

\begin{conjecture}
\label{LachaudConj}
If $X \subset \PP^{n}$ is a projective algebraic set defined
over $\Fq$ of dimension $m \geq n/2$ and degree $d \leq q + 1$ which is
a i.t. complete intersection, then
$$
\card{X(\Fq)} \leq d \pi_{m} - (d - 1)\pi_{2m - n}
= d(\pi_{m} - \pi_{2m - n}) + \pi_{2m - n}.
$$
\end{conjecture}

Conjecture \ref{LachaudConj} has just been proved by Couvreur \cite{couvreur}. We shall be 
content here to specify that there are two cases where it is easy to verify this conjecture:

--- The conjecture is true if $X$ is of codimension $1$. This is \emph{Serre's inequality} \cite{serr}: if
$X$ is an hypersurface of dimension $m$ and of degree $d \leq q + 1$, then
$$\card{X(\Fq)} \leq d q^{m} + \pi_{m - 1}.$$
--- The conjecture is also true if $X$ is a union of linear varieties of the same dimension. 
Precisely, assume that $X$ is the union of $d$ linear varieties $Y_{1}, \dots, Y_{d}$ of 
dimension $m \geq n/2$. We prove the inequality by induction on
$d$. Write $Y_{i}(\Fq) = Y_{i} \ (1 \leq i \leq d)$ for brevity. If $d = 1$
then
$$\card{Y_{1}} = \pi_{m} = (\pi_{m} - \pi_{2n - m}) + \pi_{2n - m},$$
and the assertion is true. Now if $Y_{1}$ and $Y_{2}$ are two linear
varieties of dimension $m$, then $\dim Y_{1} \cap Y_{2} \geq 2m - n$. Hence, for $d > 1$,
$$
\card{Y_{d} \cap (Y_{1} \cup \dots \cup Y_{d - 1})} \geq \pi_{2m - n}.
$$
Now note that
$$
\card{Y_{1} \cup \dots \cup Y_{d}} = \card{Y_{1} \cup \dots \cup Y_{d - 1}}
+ \card{Y_{d}} - \card{Y_{d} \cap (Y_{1} \cup \dots \cup Y_{d - 1})}.
$$
If we apply the induction hypothesis we get
\begin{eqnarray*}
\card{Y_{1} \cup \dots \cup Y_{d}} & \leq &
(d - 1)(\pi_{m} - \pi_{2m - n}) + \pi_{2m - n} + \pi_{m} - \pi_{2m - n}\\
& = & d(\pi_{m} - \pi_{2m - n}) + \pi_{2m - n},
\end{eqnarray*}
which proves the desired inequality.

%%%%%%%%%%%%%%%%%%%%%%%%%%%%%%%%%%%
%%%% Back matter

%\section{References}

%\bibliographystyle{plain}
%\bibliography{stdcodinglib_sbrr}
%bibliographie en dur pour la version finale

\end{document}